\documentclass[11pt]{article}
\usepackage{amsmath}
 \usepackage{amscd}
 \usepackage{amssymb}
 \usepackage{theorem}

\usepackage{color}

 \usepackage{graphicx}
\usepackage{amsmath}
 \usepackage{amscd}
\usepackage{amssymb}
  \usepackage[all]{xy}

\newtheorem{TEO}{Theorem}[section]
\newtheorem{PROP}[TEO]{Proposition}

\newtheorem{DEF}[TEO]{Definition}
\newtheorem{COR}[TEO]{Corollary}

\theorembodyfont{\normalfont}
\newtheorem{EX}[TEO]{Example}

\newtheorem{REM}[TEO]{Remark}

\theorembodyfont{\normalfont}

\newcommand\Oh{{\mathcal O}}

\newcommand\sF{{\mathcal F}}
\newcommand\sG{{\mathcal G}}
\newcommand\sH{{\mathcal H}}
\newcommand\sI{{\mathcal I}}

\newcommand\sK{{\mathcal K}}
\newcommand\sL{{\mathcal L}}
\newcommand\sM{{\mathcal M}}

\newcommand\om{\omega}

\newcommand\dual{\mathrel{\raise3pt\hbox{$\underline{\mathrm{\thinspace d
\thinspace}}$}}}

\newcommand\iso{\cong}

\newcommand\into{\hookrightarrow}

\newcommand\C{\mathbb C}
\newcommand\proj{\mathbb P}
\newcommand\Ka{\mathbb K}

\newcommand\length{\operatorname{length}}

\newcommand\im{\operatorname{Im}}

\newcommand\Ann{\operatorname{Ann}}

\newcommand\Cliff{\operatorname{Cliff}}
\newcommand\gon{\operatorname{gon}}

\newcommand\Hom{\operatorname{Hom}}

\newcommand\sHom{\operatorname{{\mathcal H}{\it om}}}

\newcommand\Pic{\operatorname{Pic}}

\newcommand\Sing{\operatorname{Sing}}

\newcommand\cliff{\operatorname{Cliff}}
\newcommand\Syz{\operatorname{\Syz}}
\newenvironment{proof}[1][]{\noindent\textbf{Proof#1}.  }{{\hfill $\blacksquare$}}

\begin{document}

\title{ Clifford index for reduced  curves}
%\thanks{This research was partially supported  by Italian MIUR through PRIN 2010 project  ``Geometria delle variet\`a algebriche e dei loro spazi di moduli". }}
\author{Marco Franciosi}
\date{ }

\maketitle

\begin{abstract}
%Let $C$ be a  reduced curve.
We extend the notion of Clifford index  to reduced curves  with planar singularities  by considering 
rank 1 torsion free sheaves. We investigate the behaviour of the Clifford index with respect to the 
combinatorial properties   of the curve and we show that Green's conjecture holds
for  certain classes of curves given by  the union of two irreducible components.

\hfill\break	
{\bf keyword:} Algebraic curve,   Clifford index, Green's conjecture 

\hfill\break  {\bf Mathematics Subject Classification (2010)} 14H20,  14C20, 14H51
\end{abstract}

%\newpage

%\tableofcontents
\section{Introduction}
Clifford index for  smooth curves has been introduced by H. Martens in \cite{Martens1968}  (see also \cite{GL}),
and many authors  investigated its relation with the geometry of smooth curves. 
If $C$ is a smooth curve and $\sL$ is an invertible sheaf, then the Clifford index of $\sL$ is 
$ \Cliff(\sL)= \deg(\sL)- 2 h^0(C,\sL) +2$ and the Clifford index of $C$ is
$$\Cliff(C)=\min_{\sL\in \Pic(C)}\{ \deg(\sL)- 2 h^0(C,\sL) +2 \  : \  h^0(C,\sL)\geq 2, h^1(C,\sL)\geq 2\} .
 $$
For a smooth curve it is always $\Cliff(C)\geq 0$, with  equality holding  only for
hyperelliptic curves, $\Cliff(C)=1$ if and only if $C$ is trigonal or plane
quintic, and $\Cliff(C)= 2 $ if and only if $C$ is tetragonal, or plane sextic  (see \cite{Martens1982} for a  further analysis). 
Indeed  the Clifford index is intimately, but not completely,  related to the gonality since it is 
$\Cliff(C) +2 \leq \gon(C) \leq \Cliff(C)+3$ (see \cite{CM91}).

Caporaso in \cite{CAP}  studied the Clifford index of invertible sheaves on  semistable  curves finding  interesting connections with the combinatorial properties of the curve
and pointing out the problems that can arise if the curve has disconnecting nodes.
Tenni and the author  in \cite{FrTe1} proved a generalisation of Clifford's theorem for singular  curves, either reduced with planar singularities or lying on a smooth surface,  studying
rank one torsion free sheaves of the form  $ \sI_S \omega_C$, where $S$ is a zero dimensional scheme and  $\omega_C$ is the canonical sheaf.  

In this  paper
we  consider  reduced  curves with planar singularities (e.g., semistable curves)  and we study   nef torsion free sheaves of rank 1  whose degree is bounded from above by the degree of the
 canonical sheaf  $\omega_C$.  We recall that these curves are always Gorenstein.

Notice that, for a curve $C$  with many components the behaviour of the sections of a torsion free sheaf may be rather complicated,
 hence the Clifford index too.
 Nevertheless %under some natural assumptions on its connectedness,
  it is possible to find an estimate for the Clifford  index 
 and to  show some geometric  relations with the canonical ring of the curve.  Indeed,
 given a reduced curve $C= C_1 \cup \cdots \cup C_n$, 
% denoting as usual  by $\omega_C$  a canonical divisor of $C$, % in section 3
and a  rank 1 torsion free sheaves $\sF$ on $C$  such that 
\begin{equation}\label{conditionF}
  0 \leq \deg[\sF_{|C_i}] \leq \deg {\omega_C}_{|C_i} \ \ \ \forall C_i ,  \ i=1, \cdots, n \end{equation}
 we   %extend the notion of   Clifford  index of such sheaf  by    setting  
 set  $\Cliff(\sF):= \deg(\sF)- 2 h^0(C,\sF) +2 $. Taking in account all the sheaves that contributes to the Clifford index  (see  Definition \ref{CliffC})    
 %we restrict our attention to  those which satisfy $h^0 (\sF) \geq2, h^1(\sF)\geq 2$, i.e. 
we give the following definition  of Clifford index for a reduced curve $C$ 
$$
\begin{array}{rl}
 \Cliff(C): = \min\{ \Cliff(\sF) \ : &  \sF \mbox{ rank 1 torsion free sheaf s.t. } \\
 &  \sF  \text{
verifies } (\ref{conditionF}) \ ;   h^0(\sF) \geq 2,  \ h^1(\sF) \geq 2
 \}.
 \end{array}
$$
In Section  3   we prove that such minimum does exist and 
  we show 
 lower and upper bound for such index,  investigating  its relation with the combinatorial properties of the curves, in particular   {\em  $m$-connectedness}. 
Recall that 
 a curve  $C$    is {\em  $m$-connected}  if  for any proper  decomposition $C=A \cup B$, it is $A\cdot B   := \deg ({\omega_C}_{| B})- (2p_a(B)-2) 
  \geq m$ (cf. \cite[\S 3]{CFHR}). 
  
More precisely  we prove  firstly  that $ \Cliff(C)$  can be negative  if $C$ is not 4-connected,   
bounded from below by $ -n +1$, where $n$ is the number of irreducible components of $C$ (see Prop. \ref{lower_bound}).
We show also  that such bound is sharp providing examples given by  chains of curves (see Example \ref{chain})
and we provide an  example of a 3-connected curve  $C$ with canonical sheaf very ample but  $ \Cliff(C) = -1$ (see Example \ref{3con}).  
On the contrary  in  Thm. \ref{4-connected} we  show  that $ \Cliff(C) \geq 0$   if $C$ is 4-connected  and  in Thm.  \ref{clifford invertible} we prove that 
for every invertible sheaf $\sL$  it is $ \Cliff(\sL) \geq 0$,  independently from the connectedness of  the curve.
Finally  in Theorem  \ref{prop:Clifford}
  we show   the following  constraints  given by the numerical connectedness of $C$:  if $C$ does not contain rational components and  it is  $m$-connected but  $(m+1)$-disconnected
  (that is, there is a decomposition $C=A\cup B$ with $A\cdot B =m$)  then 
%
%Indeed in Proposition \ref{prop:Clifford} we obtain the following  upper bound for the Clifford  index of a $m$-connected curve $C$: 
$$   \Cliff(C) \leq \min  \big\{m-2, \big[\frac{p_a(C)-1}{2}\big] \big\}. $$
We remark that our results can still  be applied to irreducible curves with planar singularities. In particular for an irreducible curve $C$  it is
always $ \Cliff(C)\geq0$, with  equality holding iff $C$ is hyperelliptic.

\hfill\break
To show that the above introduced Clifford index ha a geometrical meaning  in Section \ref{green} we give a   proof  of Green's conjecture for
a  $m$-connected  curve obtained glueing together two smooth curves.  To be more precise
we consider a stable curve  $C=C_1\cup C_2$  given by the union of  an irreducible  smooth  general
  curve  $C_1$  of genus $g_1$  and an irreducible smooth curve $C_2$   of genus $g_2 $,
 meeting transversally in $m$ distinct points
 $\{x_1, \cdots, x_m\}$.   For such curve $ C$  if  $ 4 \leq m \leq   \frac{g_1 +1}{2}  $ and $g_2\geq 1$
then  we show that $\Cliff(C)=m-2$ and Green's conjecture holds for $C$, i.e.,  $\sK_{p,1}(C,\omega_C)=0$ iff $p\geq p_a(C)-\Cliff(C)-1$
  (where  $\sK_{p,1}(C, \omega_C)$ denotes  the $p$-th   Koszul  group  with value in $\omega_C$ (see  Green's paper \cite{Gr}).
This result is only a modest  novelty,  since it is based on fundamental  results of Voisin in \cite{V1,V2} 
and  Aprodu in  \cite{ap02}, but we hope  it should be helpful  %in order to investigate the role of numerically connectedness
in studying curves with many components, e.g.,  stable curves. 

A second application of our results can be found in the paper  \cite{Ar-Br}, where the authors,
in order to characterize  Brill-Noether-Petri curves,  analyze
 the Petri homomorphism for rank 2 vector bundles on a (not necessarily smooth)  curve $C$ 
 using some results on the Clifford index.

\subsection*{Acknowledgements} 
The author is grateful for support by the PRIN project  2015EYPTSB$\_$010  ``Geometry of Algebraic Varieties'' of italian MIUR. 

The author would like to thank Elisa Tenni for deep  and stimulating discussions on these arguments.

\section{Notation and preliminary results}

%\subsection{Notation}
 We work over  an algebraically closed field $\Ka$ of characteristic $ 0$.

Throughout this paper  $C=C_1\cup\cdots \cup C_n$ will denote a reduced  curve with planar singularities.
 The $C_{i}$'s  are    the
 irreducible components of $C$.

\hfill\break  A subcurve $B\subseteq C$ is a curve of the form $B=C_{i_1} \cup\cdots C_{i_k}$ with $\{ i_1, \ldots, i_k\} \subset \{1,\ldots, n\}$.
For every subcurve $B\subseteq C$
 $\om_B$  denotes the   canonical  sheaf of $B$ (see \cite{Ha}, Chap.~III, \S7), $K_B$ denotes a canonical divisor so that $\Oh_B(K_B) \iso \omega_B$ 
 and
$p_a(B)$ the arithmetic genus of $B$, $p_a(B)=1-\chi(\Oh_B)$.  

Notice that by our assumptions every $B\subseteq C$ is Gorenstein (i.e.,  $\omega_B$ is invertible.)

\hfill\break
A decomposition $C= A\cup B$ means $A=C_{j_1} \cup\cdots \cup C_{j_h}$,
$B=C_{i_1} \cup\cdots  \cup C_{i_k}$ such that $\{ j_1, \ldots, j_h\}  \cup \{ i_1, \ldots, i_k\}= \{1,\ldots, n\}$.

For a given decomposition $C=A\cup B$, 
we will use the  following standard exact sequences:
 \begin{equation}0 \to \Oh_A(-B) \to \Oh_C \to \Oh_B \to 0,  \end{equation}
 \begin{equation}\label{canonical splitting} 0 \to \omega_A  \to \omega_C \to {\omega_C} _{|B} \to 0. \end{equation}
where  %$\Oh_A(-B) \iso \Oh_A \otimes \Oh_X(-B)$ if $C$ is contained in a smooth surface  $X$ (cf.  \cite[Chapter 3]{miles}) and corresponds to
$ \Oh_A(-B) \iso \sI_{A\cap B} \cdot \Oh_A$.

 If $C= A\cup B$ is a decomposition of $C$ then 
 the intersection product $A\cdot B$  is defined as follows 
$$A \cdot B =
 \deg_B(\omega_C)- (2p_a(B)-2) . $$
  A
  curve $C$ is  {\em
$m$-con\-nected} if
$A \cdot B \geq m $
for every effective decomposition  $ C=A \cup B$  (cf.   \cite{CFHR} for a detailed analysis on Gorentein curves).
An $m$-connected curve   $C$ is  said to be $(m+1)$-disconnected if 
 there is a decomposition $C=A\cup B$ with $A\cdot B =m$. 

\hfill\break 
For a decomposition $C= A\cup B$ we will use frequently  the key formula (cf. \cite[Exercise V.1.3]{Ha})
 \begin{equation}\label{genere A+B}  p_a(C)= p_a(A)+p_a(B) + A\cdot B -1. \end{equation}

\hfill\break
Let $\sF$ be a rank one torsion free sheaf on $C$.   For every subcurve  $B\subseteq C$ the degree of $\sF$ on $B $
can be defined  by the formula  
 $ \deg \sF_{|B}=\chi(\sF_{|B})-\chi(\Oh_B) $. 
 A torsion free sheaf  $\sF$ is said to be nef  if $\deg \sF_{|B} \geq 0$ for every $B\subseteq C$.

 %Let $C=C_1\cup\cdots \cup C_n$. For each $i$ the natural inclusion map $\epsilon_i:\Ga_i \rightarrow C$ induces a map
%$\epsilon_i^{\ast}:\sF \rightarrow \sF_{|\Ga_i}$. We denote by  $d_i= \deg (\sF_{|\Ga_i}) = \deg_{\Gamma_i} \sF$ the degree of $\sF$ on each  irreducible component,
 %and by  ${\bf d}:=(d_1, ..., d_s)$  the
% {\em multidegree  of $\sF$ on $C$}.  If  $B $ is a subcurve of $C$, by $\mathbf
%d_{B}$  we mean the multidegree of $\sF_{|B}$.
%We remark that there exists a natural partial ordering given by the multidegree.

%  A torsion free sheaf  $\sF$ is said to be nef (numerically eventually free) if % for every $B\subseteq C$  it is $\deg \sF_{|B} \geq 0.$ 
%  If  $C= \bigcup_{i=1}^n C_i}$ it is equivalent to ask   $d_i \geq 0 $ for every $i$ . 

\hfill\break 
 A {\em cluster} $S$ of {\em degree}
$\deg S=r$ is a \hbox{$0$-dimensional} subscheme with
$\length\Oh_S=\dim_k\Oh_S=r$. % The multidegree of $Z$ is defined as the opposite of the multidegree of $\sI_Z$. We consider the empty set as the degree 0 cluster.
 A cluster $S \subset C$ is {\em subcanonical} if the space $H^0(C, \sI_S \omega_C)$ contains a generically invertible section, i.e. a section $s_0$ which does not vanish on any subcurve of $C$.
Equivalently $S$ is subcanonical if there exists an ejective map $\Oh_C \into \sI_S\omega_C$
(see \cite[\S 2.3]{FrTe1}).

Given a subcanonical cluster $S$, we define  
its  residual Cluster $S^{\ast}$  with respect to a generic invertible section $s_0 \in H^0(C, \omega_C)$      by 
  the following exact sequence
 $$\xymatrix{
 0  \ar[r]  &  \sHom(\sI_S \omega_C, \omega_C) \ar[r]^{\alpha}   & \sHom(\Oh_C, \om_C)  \ar[r]   & \Oh_{S^{\ast}}  \ar[r] & 0
 }$$
 where the the map  $\alpha$ is defined by $\alpha(\varphi): 1 \mapsto \varphi(s_0)$. See  \cite[Section  2]{FrTe1} for the definition and main properties.

 %\hfill\break Finally, a  curve  $C$ is {\em honestly hyperelliptic}  if there exists a finite morphism \mbox{$\psi\colon C\to\proj^1$} of degree $2$ (see    \cite[\S3]{CFHR} for a detailed treatment).

%\subsection{Preliminary results } %on the canonical divisor}
%Let  $\om_C$ be the   dualizing sheaf of $C$ and  $\omega_C$  a canonical divisor on $C$.
\hfill\break 
In the following theorem we summarise some basic results  proved in
 \cite{CFHR}   on the relations   of  $m$-connectedness with  the behaviour of the canonical sheaf  $\omega_C$.
   For a general treatment see  \cite[\S 2, \S 3]{CFHR} and \cite{CF}.

 \begin{TEO}\label{thm:curve} Let $C$ be a
Gorenstein curve,  and $\omega_{C}$ the   canonical sheaf  of $C$. Then
  \begin{enumerate}
	 \renewcommand\labelenumi{(\roman{enumi})}
	
 \item If $C$ is 1-connected then $H^{1}(C,\omega_C) \iso \Ka$.

 \item If $C$ is 2-connected and $C\not \iso \proj^{1}$ then $|\omega_{C}|$ is base point free. To be more precise,  $P$ is a base point for $|\omega_C|$ if and only if there exist a decomposition 
$C=C_1 \cup  C_2$ such that $C_1 \cdot C_2=1$ and  $P$ is a smooth point for each $C_i$ satisfying  ${\omega_{C}}_{|C_i}\cong {\omega}_{C_i}( P)$.

\item If $C$ is 3-connected and $C$ is not honestly hyperelliptic (i.e.,
there does not exist a finite
morphism $\psi\colon C\to\proj^1$ of degree $2$) then $\omega_{C}$ is very ample.
 \end{enumerate}
 \end{TEO}
(cf.     \cite[Thm. 1.1,  Thm. 3.3,  Thm. 3.6]{CFHR} and   \cite[Proposition 2.4]{CF}).\\

\begin{REM} Notice that for reduced curves the above implications are actually equivalences. Indeed, $(i)$ is obvious; $(ii)$ follows from the fact that a disconnecting point is necessarily a base point for $|\omega_C|$; $(iii)$ follows since, given a decomposition $C=A\cup B$ with $A \cap B=\{P,Q\}$, then $|\omega_C|$ does not separate the 2 points. 
See also \cite[Proposition 2.4]{CF} for a detailed  analysis of the base points of $|\omega_C|$ on 2-disconnected curves.
\end{REM}

\section{Clifford index of reduced   curves}\label{section:Clifford}

\subsection{ Clifford index of rank 1 torsion free sheaves}
In this section we extend the notion of Clifford index taking in account  nef  rank 1
torsion free sheaves
 whose multidegree  is bounded from above by the degree of the canonical sheaf
$\omega_C$.

\begin{DEF}\label{def_index} Let $C=C_1\cup\cdots \cup C_n$ be a connected reduced  curve with planar singularities 
and let $\sF$ be a nef  rank 1
torsion free sheaf. 

The Clifford index of $\sF$  is %\marginpar{the cliff index as }in the following way  %in the usual way:
$$\Cliff(\sF):= \deg(\sF)- 2 h^0(C,\sF) +2 $$

\end{DEF}

%\begin{REM} A rank 1 torsion free sheaf $\sF$ is isomorphic to a sheaf  $\sI_S L$ where $S$ is 0-dimensional scheme  and  $L$ is an invertible sheaf. \end{REM}
First of all let us consider the case $\sF = \sI_S \omega_C$, where $S \subset C$ is a subcanonical cluster, i.e $S$ is a 0-dimensional scheme such that  $H^0(C, \sI_S \omega_C)$ contains a generically invertible section.
\begin{PROP}\label{lower_bound}Let $C=C_1\cup\cdots \cup C_n$ be a connected reduced  curve with planar singularities 
and let  $S$ be a subcanonical cluster.
 Then
 $$ \Cliff(\sI_S \omega_C) \geq - n +1 .$$
\end{PROP}
\begin{proof}We argue by induction on the number of irreducible components $n$.

If the curve $C$ is irreducible or reducible and 2-connected it is a straightforward consequence of \cite[Thm. 3.8]{FrTe1}.

If $C$ is connected but 2-disconnected, then we may take a decomposition $C=A\cup B$, with A, B connected curves such that  $A\cdot B=1$, i.e.
$A\cap B =  \{P\}$  a point  which is smooth for both.
 Let  $n_A$  be the number of irreducible components of $A$ and  $n_B$ be  the number of irreducible components of  $B$, 
so that $n=n_A+n_B$.

Let $S$ be a subcanonical cluster, i.e., assume that $H^0(C, \sI_S \omega_C)$ contains  a section $s_0$ which does not vanish on any subcurve of $C$
  % and let $S_A =  S\cap A$, $S_B=S\cap B$. 
  and 
 consider  the intersection point  $P$. Notice that 
  $P$ is a smooth point for both curves and it  is a base point for the system $|\omega_C|$ by Theorem  \ref{thm:curve}. 
Without loss of generality we may assume  that $P \cap S \neq \emptyset$. Indeed, if this is not the case, we may consider a residual cluster  $S^{\ast} \in H^0(C, \sI_{S} \omega_C)$  with respect to $s_0$ (see \S 2). Since $P$ is a base point for $|\omega_C|$, then  $P$ must intersect either $S$ or $S^{\ast}$. Serre duality implies that  the Clifford index of $S$ and $S^{\ast}$ coincide (see \cite[Remark 2.13]{FrTe1}), thus we may work with the cluster  which contains $P$.

Since $P$ is a smooth point for both the curves and $C$ has planar singularities, we have the isomorphisms of invertible sheaves $  {\omega_C}_{| A} \cong \omega_A(P)$ and    ${\omega_C}_{| B}\cong \omega_B(P)$.   Whence, being $P \cap S \neq \emptyset$,   there exists a cluster $T_A $ on $A$, resp.  a cluster  $T_B$ on $B$,  such that $ (\sI_S \omega_C)_{| A} \cong {\sI}_{T_A} \omega_A$, resp. 
$ (\sI_S   \omega_C)_{| B} \cong {\sI}_{T_B} \omega_B$.  Notice that  $\deg(\sI_{T_A} \omega_A) + \deg(\sI_{T_B}  \omega_B) = \deg(\sI_S  \omega_C)$. 
 Moreover they are  
 subcanonical, since a generically invertible section in $H^0(C, \sI_{S} \omega_C)$ restricts to a generically invertible section in $H^0(A, \sI_{S} {\omega_C}_{|A}) = H^0({\sI}_{T_A} \omega_A)$,  and  similarly on $B$.
 % Notice that by adjuction it is   $\deg (T_A) = \deg (S \cap A) -1$, $\deg (T_B) = \deg (S \cap B) -1$. 
 Therefore by induction we may assume  $ \Cliff(\sI_{T_A} \omega_A) \geq -n_A +1$ and $  \Cliff(\sI_{T_B} \omega_B)  \geq   -n_B  +1.$
 
 Consider now  the Mayer-Vietoris sequence
\begin{equation}\label{MV}0 \to \sI_S \omega_C  \to  \begin{array}{c}  {\sI}_{T_A}\omega_A \\ \oplus \\  {\sI}_{T_B}\omega_B  \end{array}
   \to \Oh_P \to 0 .  \end{equation}
   
   Firstly assume  that 
   $H^0(A, {\sI}_{T_A}\omega_A)  \oplus  H^0(B, {\sI}_{T_B}\omega_B)  \to \Oh_P $  is onto.  Notice that  this holds  when 
   $\Cliff(\sI_{T_A}  \omega_A)$  is  minimum  since 
    by  \cite[Lemma 2.19]{FrTe1} the restriction map   $H^0(A, {\sI}_{T_A}\omega_A)  \to \Oh_P $ is surjective. 
      In this case it is $ h^0(C,  \sI_S \omega_C) = h^0(A, {\sI}_{T_A}\omega_A) + H^0(B, {\sI}_{T_B}\omega_B) -1$, hence 
    a straightforward computation yields  $ \Cliff(\sI_S \omega_C) \geq  \Cliff(\sI_{T_A}  \omega_A) + \Cliff(\sI_{T_B}  \omega_B) $.  Therefore, by induction we have 
    $$ \Cliff(\sI_S  \omega_C) \geq \Cliff(\sI_{T_A}  \omega_A) + \Cliff(\sI_{T_B}  \omega_B)  \geq  -n_A -n_B  +2 = -n+2$$    
   
   If the above map (\ref{MV})  is not surjective on global sections then  in particular $H^0(A, {\sI}_{T_A}\omega_A)  \to \Oh_P $ is not onto and by  \cite[Lemma 2.19]{FrTe1}
   $\Cliff(\sI_{T_A}  \omega_A)$ is not minimum, i.e.,  by induction we may assume $ \Cliff(\sI_{T_A}  \omega_A) \geq -n_A  +2 $. In this case 
  it is   $ h^0(C,  \sI_S \omega_C) = h^0(A, {\sI}_{T_A}\omega_A) + H^0(B, {\sI}_{T_B}\omega_B) $, hence 
 $ \Cliff(\sI_S  \omega_C) \geq  \Cliff(\sI_{T_A}  \omega_A) + \Cliff(\sI_{T_B}  \omega_B) -2 $
   and  we get $$ \Cliff(\sI_{S}  \omega_C)\geq \Cliff(\sI_{T_A}  \omega_A) + \Cliff(\sI_{T_B}  \omega_B)  -2  \geq -n_A+2  -n_B+1 -2 =  -n+1.$$
\end{proof}

\begin{TEO}\label{cliff=}
Let $C=C_1\cup\cdots \cup C_n$ be a connected reduced  curve with planar singularities.
 Then  the following numbers exist and coincide:
\vspace{ 2 mm}
%\begin{enumerate}

$\begin{array}{ll}
\hspace{ - 8 mm} {\bf (1)} \ \  \min \{ \Cliff(\sF) \ : &  \sF \mbox{ rank 1 torsion free sheaf s.t. } \\
 & 0 \leq \deg[\sF_{|C_i}] \leq \deg {\omega_C}_{|C_i} \   \text{
 for  } \ i=1, \cdots, n ; \\ &  h^0(\sF) \geq 2,  \ h^1(\sF) \geq 2
 \}
 \end{array}$

$\begin{array}{ll}
\hspace{ - 8 mm} {\bf (2)} \ \ \min \{\Cliff(\sI_S \omega_C)   \ : & \  S \subset C  \text{ subcanonical  cluster s.t.  }   \\ 

& h^0(C,\sI_S \omega_C) \geq 2,  \ h^1(C, \sI_S \omega_C) \geq 2\}
 \end{array}$

\end{TEO}

\begin{proof}
By Proposition \ref{lower_bound} the second minimum exists. It is moreover obvious that the second set is included in the first, thus such minimum is bigger than or equal to the infimum of the first set. 

To conclude the proof it is enough to prove that for every rank 1 torsion free sheaf  $\sF$   in the first set  attaining   the minimal Clifford index 
%such that 
%$$ 0 \leq \deg[\sF_{|C_i}] \leq \deg {\omega_C}_{|C_i}  \mbox{ 
% for every  }  C_i  \mbox{ and }  h^0(\sF) \geq 2,  \ h^1(\sF) \geq 2 $$
%such that $\Cliff (\sF)$ is minimum 
there exists a subcanonical cluster $T$ such that $\sI_T \omega_C \iso \sF.$ %  satyisfing  $h^0(C,\sI_S \omega_C) \geq 2,  \ h^1(C, \sI_S \omega_C) \geq 2.$
This is equivalent to  prove  that a rank 1 torsion free sheaf $ \sF$ with   minimal Clifford index   is generically invertible and moreover there exists an inclusion $ \sF \into \omega_C$ 
which is generically surjective. 

For the first  statement, assume for a contradiction that 
$\sF$ itself is not generically invertible  and  let $B\subset C$ be the maximal subcurve of $C$ such that every section
$s \in H^0(C, \sF)$  vanishes identically on  $B$. Consider the decomposition $C=A\cup B$.
Then by the standard exact sequence
$$ 0 \to \sF_{|A}(-B) \to  \sF \to \sF_{|B} \to 0 $$
we obtain 
$$ H^0(A,\sF_{|A}(-B)) \iso H^0(C, \sF)$$
$$0 \to H^0(B, \sF) \to  H^1(A,\sF_{|A}(-B)) \to H^1(C, \sF) \to H^1(B, \sF)\to 0.$$
We take the sheaf $\sG \iso \sF_{|A}(-B) \oplus \Oh_{B}(A)( -A) $. Notice that  $\sG$ is a rank 1 torsion free sheaf and moreover 
it is $0 \leq \deg_{C_i} \sG \leq \deg_{C_i} \sF$  for every $C_i$ since $\sF_{|A}(-B)$ does not vanish on any subcurve of $A$ (see \cite[Rem. 4.1]{FrTe1} for details). 

Then it is immediately seen that $\Oh_C \into \sG$, i.e., $\sG$ is generically invertible,  and
 $$h^0(\sG)= h^0(B, \Oh_B) + h^0(A,\sF_{|A}(-B)) \geq   h^0(A,\sF_{|A}(-B)) +1= h^0(C,\sF) +1 \geq 3, $$ 
 $$h^1(\sG) =  h^1(B, \Oh_B) + h^1(A,\sF_{|A}(-B)) =
 h^1(B, \Oh_B) + h^1(C,\sF)+h^0(B,\sF) -h^1(B,\sF) \geq $$ $$ \geq h^1 (C,\sF) +
h^1(C,\sF) + \deg(\sF_{|B}) + h^0(\Oh_B) \geq 3. $$
Moreover  
we obtain  $\Cliff(\sG) < \Cliff(\sF)$, since $\deg (\sG) \leq \deg (\sF)$ and  $h^0(\sG) > h^0(\sF)$, which is absurd.  Therefore we have in particular 
$\Oh_C \into \sF$.

 Now  we show  that $\sF \into \omega_C$. The dual sheaf $\sHom (\sF, \omega_C)$  satisfies the same assumptions $\sF$ does and by Serre duality  has the same Clifford index.
  Hence thanks to the previous step $\Oh_C \into \sHom (\sF, \omega_C)$. In particular $H^0(C, \Oh_C) \into H^0(C, \sHom (\sF, \omega_C))= \Hom (\sF, \omega_C)$,
  that is, there is a map from $\sF$ to $\omega_C$ not vanishing on any component. 
    By automatic adjunction (\cite[Proposition 2.4]{CFHR}) we  conclude that 
    $\sF \iso \sI_T \omega_C$, for some suitable 0-dimensional scheme $T$.

 \end{proof}
 
 \subsection{Clifford index of curves}
 The above theorem allows us to introduce the following  notion of Clifford index for a reduced  curve.
\begin{DEF}\label{CliffC} Let $C=C_1\cup\cdots \cup C_n$ be a connected reduced  curve with planar singularities.
The Clifford index of  \ $C$ is
$$
\begin{array}{rl}
 \Cliff(C): = \min\{ \Cliff(\sF) \ : &  \sF \mbox{ rank 1 torsion free sheaf s.t. } \\
 & 0 \leq \deg[\sF_{|C_i}] \leq \deg {\omega_C}_{|C_i} \   \text{
 for every } C_i \ ; \\ &  h^0(\sF) \geq 2,  \ h^1(\sF) \geq 2
 \}
 \end{array}$$
\end{DEF}
As in the smooth case, we say that a rank 1 torsion free sheaf  $\sF$ contributes to the Clifford index of the curve $C$ if $h^0(C,\sF) \geq 2$ and $   h^1(C,\sF) \geq 2$. 

For 4-connected curves   the Clifford index 
is always non-negative as can be seen by the following result.

\begin{TEO}\label{4-connected} If $C$ is a  4-connected  reduced curve with planar singularities then $\Cliff(C) \geq 0$ and it is 0 if and only if $C$ is honestly hyperelliptic. \end{TEO}
\begin{proof} By \cite[Theorem B ]{FrTe1} if $C$ is 4-connected then  for every rank one torsion free sheaf  $\sF$ we have $h^0(C,\sF )\leq \frac{\deg{\sF}}{2}+1$.

Moreover  the above mentioned theorem shows that if equality holds then $\sF \cong \sI_T \omega_C$,  where $T$ is a subcanonical cluster and, as in the smooth case,   either 
$T=0,\, \omega_C$
or   $C$ is honestly hyperelliptic and $T$ is a multiple of the honest $g_{2}^{1}$.

\end{proof}
\begin{COR}\label{irreducible} If $C$ is an irreducible curve  with planar singularities then $\Cliff(C) \geq 0$ and it is 0 if and only if $C$ is  hyperelliptic. \end{COR}

If $C$ has many components,  numerical $ m$-connectedness plays a relevant role in the  the computation of the Clifford index.
Indeed  we have the following

\begin{PROP}\label{decomposition}
Let $C=C_1\cup\cdots \cup C_n $  be a  $m$-connected reduced curve with planar singularities.
  
  Assume there exists a decomposition $C=A\cup B$ such that $A\cdot B=m$, $p_a(A)\geq 1$, $p_a(B)\geq 1$.
  Then there exists an invertible sheaf $\sF$  such that $ \Cliff(\sF)=m-2$, $h^0(\sF) \geq 2, h^1(\sF) \geq 2.$

\end{PROP}
\begin{proof}
Consider the decomposition $C=A\cup B$ with  $A\cdot B=m$.  Notice that  $A$ and $B$ are numerically connected
by minimality  of $m$ (see   \cite[Lemma 2.8]{FrTe1}) and moreover by our assumptions it is $p_a(A)\geq 1 , \  p_a(B) \geq 1 $. 

If $m=1$ then by Thm. \ref{thm:curve} there exists a base point $P$ for the canonical system. Thus $\Cliff(\sI_P  \omega_C)=-1$ and we may conclude. 

From now on we  assume that $m \geq 2$ and in particular $|\omega_C|$ is base point free.
Choose a generic $s\in H^0(A, {\omega_C}_{|A}) $ and take the effective divisor $\Delta = div(s)$. Since $|\omega_C|$ is base point free, $\Delta$ is the union of smooth points and moreover by our construction we may assume  $\Delta \cap B=\emptyset$.
 Consider the invertible sheaf  $\sF:=\Oh_C(\Delta)$. It is 
$$\sF_{|A}\iso {\omega_C}_{|A} \ , \ \ \sF_{|B}\iso \Oh_B.  $$
In particular $ \sF (-B)_{|A}\iso  \omega_A$ and we have the exact sequence 
\begin{eqnarray*} 0\to H^0(A, \omega_A) \to H^0(C, \sF) \to H^0(B, \Oh_B) \to  \\
\to H^1(A, \omega_A) \to H^1(C, \sF) \to H^1(B, \Oh_B) \to 0.
\end{eqnarray*} 
But $H^0(C, \sF)$ does not vanish on $B$ by our construction, 
whence
$$h^0(C,\sF)=h^0(B, \Oh_B)+h^0(A, \omega_A)=1+p_a(A)\geq 2$$
$$h^1(C,\sF)=h^1(B, \Oh_B)+h^1(A, \omega_A)=1+p_a(B)\geq 2$$
since both $A$ and $B$ are numerically connected. 

Finally  by the above computation  we get 
\begin{eqnarray*}\Cliff (\sF)&=&\deg_C(\Oh_C(\Delta))-2h^0(C,\Oh_C(\Delta))+2=\\&=&  2p_a(A)-2+m-2\cdot(1+p_a(A))+2= m-2. \end{eqnarray*}

\end{proof}

As an immediate consequence we obtain the following theorem.
\begin{TEO}\label{prop:Clifford}
Let $C=C_1\cup\cdots \cup C_n $  be a  connected  reduced curve with planar singularities.
  Assume $C_i \neq \proj^1$ for every $i=1,\cdots, n$.

If $C$ is  $m$-connected but  $(m+1)$-disconnected
  (that is, there is a decomposition $C=A\cup B$ with $A\cdot B =m$) then 
%Let $m:= \min \{ A\cdot B : C =A\cup B, A\neq \emptyset, B\neq \emptyset\}$.
% and let  $C=A\cup B$  be a decomposition  such that $A\cdot B = m$. % and $p_a(A) \geq 1$, $p_a(B)\geq 1.$ Then
%there exist an invertible sheaf $\sF$  which contributes to the Clifford index such that $\Cliff (\sF) \leq m-2$. In particular  
$$ \Cliff(C) \leq \min  \big\{m-2, \big[\frac{p_a(C)-1}{2}\big] \big\} .$$

\end{TEO}

\begin{proof}
First of all  let us show that $ \Cliff(C) \leq \min  [\frac{p_a(C)-1}{2}] $ by a degeneration argument. 

Consider  a one-parameter degeneration $f:X \to T$, where $X$ is a smooth surface and  $T$ an affine curve.
Assume that $f$  is flat and proper and there is a point $s_0\in S$ such that $f^{-1}(t_0):=C_0\iso C$,
whilst  for $t\neq 0$  $f^{-1}(t):= C_t$ is a smooth curve of genus $p_a(C)$.
 For  each integer $d$,
let $\Pic_{f}^{d}$ be the degree-$d$  relative Picard scheme
of $f$  parameterizing invertible sheaves of
degree $d$ on the fibers of $f$
 (see e.g.  \cite{grothendieck}).
  Then, for every  invertible sheaf  $\sF$ in $\Pic_{f}^{d}$ we have
 $\Cliff (\sF_{| C_t})  \leq \big[\frac{p_a(C)-1}{2}\big] $.
 Since $h^0 (\sF_{| C_t})$ and $h^1 (\sF_{|C_t})$ are semicontinuous function in $t$, if $\sF_{| C_t}$ contributes to the Clifford index for some $t$, then
$\sF_{| C_0} $ contributes to the Clifford index of $C_0$ and moreover  by definition it is
 $$ \Cliff(\sF_{| C_0}) \leq
\Cliff (\sF_{| C_t})  \leq \big[\frac{p_a(C)-1}{2}\big] .
$$

\hfill\break
Now  let us  show that $ \Cliff(C) \leq m-2$, where $m:= \min \{ A\cdot B : C =A\cup B, A\neq \emptyset, B\neq \emptyset\}$. 

Take  a proper decomposition $C=A\cup B$ with  $A\cdot B=m$.  Notice that  $A$ and $B$ are numerically connected
by minimality  of $m$ (see   \cite[Lemma 2.8]{FrTe1}) and moreover by our assumptions it is $p_a(A)\geq 1 , \  p_a(B) \geq 1 $. 
Therefore by the above Proposition \ref{decomposition} there exists an invertible sheaf 
  $\sF=\Oh_C(\Delta)$  which  contributes to the Clifford index of $C$ and whose Clifford index verifies 
  $ \Cliff(\sF)=m-2$.

\end{proof}

\begin{REM} If we restrict our attention to stable curves, it
is worth mentioning that the Clifford index, as defined in Definition
 \ref{CliffC}, is not the limit of the Clifford index of smooth curves. More precisely, if a curve $C$ is limit of smooth curves $C_t$ with $\Cliff(C_t) \leq \gamma$, then by semicontinuity we still have  $\Cliff(C) \leq \gamma$, but the converse does not hold. One can see that with a simple dimensional count. It is easy to compute that the locus of reduced $m$-connected curves has codimension $m$ in $\overline{\sM}_g$, and Theorem  \ref{prop:Clifford} shows that those curves have Clifford index at most $m-2$. On the contrary, considering the loci
$ \sM^r_{g,d}$ of smooth curves carrying a $g^r_d$   one can see that for small $m$, the locus of smooth curves having Clifford index at most $m-2$  has a far bigger codimension than $m$.

\end{REM}

\begin{REM}
Let $C$ be a $m$-connected, but $(m+1)$-disconnected,  reduced  curve of arithmetic genus $p_a(C)>0$ and let $C=A\cup B$ be a decomposition of $C$ in two connected curves of arithmetic genus $p_a(A)$, respectively $p_a(B)$,
such that $A\cdot B=m$.

By the key  formula (\ref{genere A+B}) if  $m  \leq p_a(A) + p_a(B) +2$ then it is 
$ m-2 \leq \big[\frac{p_a(C)-1}{2}\big] $.
Therefore it is easy to construct stable curves with given Clifford  just by taking $m$ satisfying the above relation.
\end{REM}

\subsection{Examples of  curves with negative Clifford index}\label{example}

In this section we are going to show two examples of curves having negative Clifford index.  The first  example shows that the inequality of Proposition \ref{lower_bound} is sharp.
The second example shows that for curves not 4-connected the geometric interpretation of the Clifford index is more subtle. 

\begin{EX}\label{chain}
Let $C= \cup_{i=0}^n C_i$ be a chain of smooth curves $C_i$ with positive genus, i.e. $C_i \cdot C_{i+1}=1$, otherwise $C_i \cdot C_j$ vanishes.

$$\xymatrix{
\Gamma_0  \ar@{-}[r] & \Gamma_1  \ar@{.}[r] & \Gamma_{n-1}  \ar@{-}[r]  & \Gamma_n \\
}$$
Let $S= \Sing{C}$. Then 
$h^0(C, \sI_S \omega_C)= \sum_{i=1}^n p_a(C_i)= p_a(C).$ Therefore we obtain 

$$\Cliff(\sI_S \omega_C)= 2 p_a(C) -2 - \deg S - 2 h^0(C, \sI_S \omega_C) +2= -n +1.$$
\end{EX}

Notice that in the above example every point in $S$ is a base point of $|\omega_C|$. 
Now,  let us point out that this is not always the case. Indeed if $C$ is 3-connected but not 4-connected  (i.e.,  there exists a decomposition $C=A\cup B$ such that $A\cdot B= 3$)
then it might happen that $ \Cliff(C) <0$ even thought $\omega_C$ is normally generated, as shown in the following example.

\begin{EX}\label{3con}
Take  $C= \sum_{i=0}^5 \Gamma_i$ and suppose that $p_a(\Gamma_i) \geq 2$ for every $i$. Suppose moreover that the intersection products are defined by the following dual graph, where the existence of the simple line means that the intersection product between the two curves is 1. 
 $$
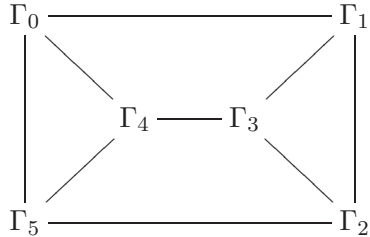
\xymatrix{\Gamma_0  \ar@{-}[rrr]  \ar@{-}[dr]  \ar@{-}[dd]&&& \Gamma_1  \ar@{-}[dl]  \ar@{-}[dd]\\
 & \Gamma_4  \ar@{-}[r]  \ar@{-}[dl]& \Gamma_3  \ar@{-}[dr] &\\
 \Gamma_5  \ar@{-}[rrr]&&& \Gamma_2
}$$
In this case by \cite[Thm. 3.3]{FrTe2} $\omega_C$ is very ample and normally generated by \cite[Thm. 3.3]{FrTe2}. 

For simplicity, assume moreover that for every $i,j,k$ it is  $\Gamma_i \cap \Gamma_j \cap \Gamma_k =\emptyset$ and 
take $S= \bigcup_{i,j} (\Gamma_i \cap \Gamma_j)$, which is a degree 9 cluster. Then it is easy to check  $h^0(C, \sI_{S}\omega_C)= \oplus_{i=0}^5 h^0(\Gamma_i, K_{\Gamma_i})$, which yields  %and that the splitting index of $S^{\ast}$ is $k=5$.
$\Cliff(\sI_S \omega_C) = -1
$.
(see \cite[Example 5.2]{FrTe1}).

%Therefore  in this example $\Cliff(C)$  is negative and $\omega_C$ is normally generated. 

\end{EX}

\subsection{Clifford index of invertible sheaves}\label{invertible-sheaves}
The following theorem shows that the Clifford index of an invertible sheaf is always nonnegative. 
\begin{TEO}\label{clifford invertible}  Let $C=C_1\cup\cdots \cup C_n$  be a reduced curve with planar singularities. 
Let $\sL$ be an invertible sheaf  such that 
$0 \leq \deg[\sL_{|C_i}] \leq \deg {\omega_C}_{|C_i} $
 for $i=1, \cdots , n$. 
%$\Oh_C \into L \into \omega_C$. 
Then
\begin{equation}\label{cliff_ridotte}
h^0(C, \sL) \leq \frac12 \deg  \sL +1  , \mbox{ i.e.,  }  \  \  \Cliff(\sL) \geq 0.
\end{equation}
\end{TEO}

\begin{proof} % First of all notice that $\Oh_C \into \sL \into \omega_C$ 
First of all we remark that we may assume $C$ to be connected since $h^0$ and $\deg$ are additive with respect to each connected component. 

Now notice that we may assume  $\sL \not\iso \Oh_C,$ $\sL \not\iso \omega_C$ and $h^0(C, \sL)\neq 0$, since otherwise  eq. (\ref{cliff_ridotte})  is obvious.  
 Take 
 $\sL$ as above such that $\Cliff(\sL)$ is minimum.   Arguing as in the proof of  Theorem \ref{cliff=}  we  conclude that $\Oh_C \into \sL \into \omega_C$, i.e., 
there exists a subcanonical Cartier divisor $S$ such that $\sL \iso \sI_S \om_C$ (see also \cite[\S 2.3]{FrTe1}).
Hence it is sufficient to show  that for every subcanonical Cartier divisor $S$,  $\Cliff (\sI_S \om_C) \geq 0$. 
 We prove this result by induction on the number of irreducible components of $C$. To simplify the notation we write $K_C-S$  for the divisor such that $\Oh_C(K_C-S)\cong
 \sI_S \om_C.$

%Notice that,  since $C$ is reduced and connected, it is 1-connected.
If $C$  is irreducible, the classical Clifford's theorem holds (see \cite[\S III:1]{ACGH}, or see \cite[Theorem A]{FrTe1} for the singular case).
 If $C$ is  2-connected the result follows from  \cite[Theorem A, case (a)]{FrTe1}.

Therefore we are left to  prove that equation (\ref{cliff_ridotte}) holds for reducible, connected but  2-disconnected curves, i.e.,  we may assume  that 
%We are going to prove the inequality by induction.
 there exist connected subcurves $C_1$ and $C_2$ such that $C=C_1\cup C_2$ and $C_1 \cap C_2$ consists of one single point $P$.  In this case 
  %$C_1$ and $C_2$  are reduced and connected  curves and  
  $P$ is a smooth point for both curves  and for $i=1,2$ we can write  ${K_C}_{| C_i} \equiv  {K}_{C_{i}}+P$ as divisors on $C_i$.  

Take  the subcanonical Cartier divisor $S$.  % and its restrictions $S_i$ to each curve. 
Arguing as in Prop. \ref{lower_bound}
 we may assume  that $P \cap S \neq \emptyset$
since otherwise we can take a  residual Cartier divisor $S^\ast $.

Let   $S_1:=S\cap C_1$ and $S_2:=S\cap C_2$.  By the above argument 
 $P \cap S_i \neq \emptyset$ for $i=1,2$, and, since $P$ is a smooth point for each $C_i$,  both the divisors 
$(S_1 -P)$ and $(S_2 -P)$ are  Cartier and  effective. Moreover they are  
 subcanonical on  both the subcurves, since a generically invertible section in $H^0(C, K_C-S)$ restricts to a generically invertible section in $H^0(C_i, K_{C_i}(-(S_i-P)))$.

 % the divisor  $S_i-P$ on each subcurve is Cartier. 

The  exact sequence (\ref{canonical splitting}) for the splitting $C=C_1\cup C_2$  can be written as follows:
%\[0 \to \omega_{C_1} \to \omega_C \to \omega_{C_2}(P)\to 0\]
%and tensorize it with $ \Oh_C(-S)$:
\begin{eqnarray*} 0 \to \omega_{C_1}(-S_1) \to \omega_C(-S) \to \omega_{C_2}(-(S_2-P))\to 0 . \end{eqnarray*}
In particular it gives rise to the inequality
\begin{eqnarray}\label{spezzamento} h^0(C, K_C-S) \leq h^0(C_1, K_{C_1}-S_1)+ h^0(C_2, K_{C_2}-(S_2-P)) . \end{eqnarray}

On $C_2$ we may apply our  induction argument obtaining   $h^0(C_2, K_{C_2}-(S_2-P)) \leq  \frac12 \deg (K_{C_2}-(S_2-P))+1$.

\hfill\break
Let us consider $H^0(C_1, K_{C_1}-S_1)$. 
Counting dimensions we have either that $h^0(C_1, K_{C_1}-S_1)=h^0(C_1, K_{C_1}-S_1+P)-1$ or that $h^0(C_1, K_{C_1}-S_1)=h^0(C_1, K_{C_1}-(S_1-P))$.

In the first case  eq.  (\ref{spezzamento})  becomes
\begin{eqnarray*} h^0(C, K_C-S) \leq h^0(C_1, K_{C_1}-(S_1-P))-1+ h^0(C_2, K_{C_2}-(S_2-P)). \end{eqnarray*}
But  $S_i-P$ are subcanonical divisors on each subcurve, hence we may apply induction on $C_1$ and $C_2$ obtaining 
\begin{eqnarray*}h^0(C, K_C-S) &\leq& \frac12 \deg (K_{C_1}-(S_1-P))+1 -1 + \frac12 \deg (K_{C_2}-(S_2-P))+1\\
&=& \frac12 \deg (K_C-S)+1. 
\end{eqnarray*}
%i.e.,  the pair $(C,S)$ satisfies Clifford's inequality (\ref{cliff_ridotte}). \\
%
%\hfill\break
In the second case $H^0(C_1, K_{C_1}-S_1)=H^0(C_1, K_{C_1}-(S_1-P))$ and  in particular  also  $S_1$ is   subcanonical on $C_1$. Therefore we may apply induction on eq.
(\ref{spezzamento}) obtaining 
\begin{eqnarray*}h^0(C, K_C-S)
& \leq& h^0(C_1, K_{C_1}-S_1)+ h^0(C_2, K_{C_2}-(S_2-P)) \\
&\leq& \frac12 \deg (K_{C_1}-S_1)+1  + \frac12 \deg (K_{C_2}-(S_2-P))+1\\
&=& \frac12 \deg (K_C-S)+\frac32.
\end{eqnarray*}
To conclude the proof it is enough to show that the above inequality is strict. 
We argue by contradiction. Assume that  $h^0(C, K_C-S) =  \frac12 \deg (K_C-S)+\frac32$. Then necessarily  
$$
\begin{array}{l}
h^0(C_1, K_{C_1}-S_1)=\frac12 \deg (K_{C_1}-S_1)+1\ ; \\ 
h^0(C_2, K_{C_2}-(S_2-P))=\frac12 \deg (K_{C_2}-(S_2-P))+1.
\end{array}
$$ 
In particular $\deg S_1$ must be even and $\deg S_2$ must be odd.
But we may switch the roles of $C_1$ and $C_2$ and conclude that $\deg S_2$ is even and $\deg S_1$ is odd, which is clearly a contradiction.
\end{proof}

\begin{REM} The above result can be extended to non reduced curves, under suitable assumptions. Indeed the above  theorem holds for 2-connected curves, whilst in the  2-disconnected
case the key point   of the proof  is the   existence of a a decomposition  $C=C_1 \cup C_2$ with $C_1.C_2=1$ 
such that:
\begin{itemize}
\item[(a)] $C_1$ and $C_2$ satisfy Clifford's inequality;%  since they are connected reduced curve; %, thus we may use induction.
\item[(b)]  $P= C_1 \cap C_2$ is a base point for $|K_C|$ and $P$ is a smooth point on $C_i$.
\end{itemize}

In order to use point (a) we do not really need that $C_i$ are reduced, just that they satisfy Clifford's inequality for some reason. E.g. 2-connected (possibly nonreduced) curves are perfectly fine.

In order to deal with  point (b)  the key fact is that by Theorem \ref{thm:curve}   $P$ is a base point for $|\omega_C|$ if and only if there exist a subcurve 
$B \subset C$ such that $\omega_{C|B}\iso {\omega}_{B}(P)$ and $P$ is smooth for $B$.
\end{REM}

\section{Green's conjecture for  certain classes of   reduced  $m$-connected curves}\label{green}

Let $C$  be a reduced  curve, 
let $\sH$ be an  invertible sheaf on $C$ and let $W \subseteq
H^{0}(C,\sH)$ be a subspace which yields a base point free system of
projective dimension $r.$

The Koszul groups $\sK_{p,q}(C,\sH,  W)$ are defined
as the cohomology at the middle of the complex
\[
\stackrel{p+1}{\bigwedge} W \otimes H^{0}(\sH^{q-1})  \longrightarrow
\stackrel{p}{\bigwedge} W \otimes H^{0}(\sH^{q})  \longrightarrow
\stackrel{p-1}{\bigwedge} W \otimes H^{0}(\sH^{q+1})
 \]
If $W=H^{0}(C,\sH) $ they are usually denoted by
$\sK_{p,q}(C, \sH)$
(see \cite{Gr} for  the  definition and main results).
The groups $\sK_{p,q}(C, \sH)$ play a significant role if $\sH$ is very ample and  normally generated  
 since in this case 
$\sK_{p,q}(C, \sH)\otimes \Oh_{\proj^n}(-p-q)$ are the terms of the resolution of the ideal sheaf    of the embedded curve (see \cite[Thm. 2.a.15]{Gr}).

If $C$ is a Gorenstein curves with planar singularities,  3-connected and not (honestly) hyperelliptic then  by 
 \cite[Thm. 3.3]{FrTe2}  $\omega_C$ is  very ample and normally generated. Therefore 
    it is worth studying  the Koszul groups $\sK_{p,q}(C, \omega_C)$. Indeed we have  the following result.

 \begin{TEO}\label{noether} Let $C$ be  a Gorenstein curve of aritmetic genus $p_a(C) \geq 3$,  with planar singularities, 3-connected and  not honestly hyperelliptic.   Then
\begin{itemize}
\item 
  $ \sK_{0,q}(C,\omega_C)=0$ for all $q>0$, i.e., $\omega_C$ is normally generated; 
\item
$\sK_{p,q}( C, \omega_C) =0$ if  $q\geq 4$;
\item
$\sK_{p,3}( C, \omega_C) \iso \C $ if  $p=g-2$, and  $\sK_{p,3}( C, \omega_C)=0$ if $p\neq g-2$;
\item
$\sK_{p,1}( C, \omega_C)^{\vee} \iso 
\sK_{g-p-2,2}(C, \omega_C) $;
\item $\sK_{p,1}( C, \omega_C)=0 \Rightarrow \sK_{p',1}( C, \omega_C)=0 \ \forall p'\geq p$;
\item $\sK_{p,2}( C, \omega_C)=0 \Rightarrow \sK_{p',2}( C, \omega_C)=0 \ \forall p'\leq p$.
\end{itemize}
\end{TEO}

\begin{proof} $ \sK_{0,q}(C,\omega_C)=0$ for all $q>0$ follows by \cite[Thm. 3.3]{FrTe2}. The remaining statements  follows
from the same the arguments used for smooth curves (see \cite[Thm. 4.3.1]{Gr}) 
and by  the duality  results  given in   \cite[Prop. 1.4]{Fr}.

\end{proof}

Taking Definition \ref{CliffC} for a generalisation of the usual Clifford index,  Green's Conjecture (\cite[Conjecture 5.1]{Gr} can be
formulated without changes, i.e., given a 4-connected not  hyperelliptic Gorenstein curve  $C$ then one may ask if 

%\begin{CONJ}[Green's conjecture for reduced curves] Let $C$ be a 4- connected Gorenstein reduced curve of genus $p_a(C) \geq 3$. Then
\[ \sK_{p,1}(C, \omega_C)= 0  \stackrel{ ? }{ \Longleftrightarrow}  p \geq p_a(C) - \Cliff(C)-1.\]
%\end{CONJ}
%
First of all notice that, as in the smooth case,  we have a non vanishing result, and hence an upper bound  on the Clifford index of a curve $C$:
%by taking the Koszul groups of its canonical ring. 

\begin{PROP}[Green-Lazarsfeld]\label{nonvanishing} Let $C=C_1\cup\cdots \cup C_n $  be a  4-connected, not honestly hyperelliptic,  reduced curve with planar singularities.  Assume $C_i \neq \proj^1$ for every $i=1,\cdots, n$.

%  and let 
%$\sF$ be  an invertible sheaf which contributes to $\Cliff(C)$.  
Then 
$$  p \leq p_a(C) - \Cliff(C) -2  \Longrightarrow \sK_{p,1}(C, \omega_C) \neq 0 $$
%In particular, if there 
% exists a decomposition $C=A\cup B$ such that  $p_a(A) \geq 1$, $p_a(B)\geq 1$, and $A\cdot B =m $, 
%then %$\Cliff(C) \leq m-2 $ and 
%$ \sK_{p_a(C)-m,1}(C,\omega_C)\neq 0$. ????????????
\end{PROP}
\begin{proof} By Theorem  \ref{4-connected} it is   $\Cliff(C)\geq 0$  and by Theorem \ref{noether}   $\omega_C$  is normally generated. 
Now by Theorem  \ref{cliff=} there exists a proper subcanonical cluster $S$ such that  $ \sI_S \omega_C$ computes the Clifford index of $C$. 

Let $S^{\ast}$ be its  residual Cluster with respect to a generic invertible section $s_0 \in H^0(C, \omega_C)$.
By  definition  it is
$ \sI_{S^{\ast}} \omega_C \iso  \sHom(\sI_S \om_C, \om_C) $ and by Serre duality it is  $H^1(C,\sI_S K_C) \dual H^0(C, \sI_{S^{\ast}} K_C)$, so  that 
$\cliff(\sI_S K_C)=\cliff(\sI_{S^{\ast}} K_C)$.  Moreover,
denoting by
$\Lambda := div(s_0)$  the  effective divisor corresponding to $s_0$ we
 have  the following exact sequence
$$ 0 \to  \Oh_C \iso \sI_{\Lambda} \om_C \to \sI_S \omega_C \to    \Oh_{S^{\ast}}  \to 0 $$
 (see  \cite[\S 2]{FrTe1} for details).

Therefore we can  consider $ \proj(H^0(\sI_S \omega_C))$ as a  $g^r_d$, where $d=\deg \sI_S \omega_C$ and $h^0(\sI_S \omega_C)=r+1$
and $\proj(H^0(\sI_{S^{\ast}} \omega_C))$   as the residual  $g^{r'}_{d'}$, where $d'=\deg \sI_{S^{\ast}} \omega_C $ and $h^0(\sI_{S^{\ast}}  \omega_C)=r'+1$.
Setting 
 $$  W_1 = \im \{ H^0(\sI_S \omega_C) \into H^0(\omega_C)\} , \ \  W_2 = \im \{ H^0(\sI_{S^{\ast}} \omega_C) \into H^0(\omega_C)\}$$  
 and 
 $$\bar{D_1}= \Ann (W_1) \subset H^0(C,\omega_C)^\vee, \ \  \bar{D_2}= \Ann (W_2) \subset H^0(C,\omega_C)^\vee $$
 we can repeat 
verbatim the argument adopted by Green and Lazarsfeld in \cite[Appendix]{Gr}   
obtaining $\sK_{ r+r' -1,1}(C,\omega_C)\neq 0.$

To conclude it is enough to see that $r+r' -1 = p_a(C) - \Cliff(C) -2$ since 
  $d'= 2p_a(C)-2 -d$ and $\cliff(\sI_S \omega_C)=\cliff(\sI_{S^{\ast}} \omega_C)=\cliff(C)$.
  
The non vanishing of $\sK_{ p,1}(C,\omega_C)$ for  every $p < p_a(C) - \Cliff(C) -2$  follows from  Theorem \ref{noether}.

\end{proof}

\begin{COR}\label{nonvanishing2 }
Let $C=C_1\cup\cdots \cup C_n $  be a  connected  reduced curve with planar singularities.
  Assume $C_i \neq \proj^1$ for every $i=1,\cdots, n$.

If $C$ is  $m$-connected but  $(m+1)$-disconnected
then $\Cliff(C) \leq m-2$ and $\sK_{p,1}(C, \omega_C)\neq  0   $ if  $ p \leq  p_a(C)  - m.$

\end{COR}

To show that our notion of Clifford index has a geometrical meaning we show that Green's conjecture holds 
 in the particular case of a stable curve consisting of two smooth components intersecting in $m$ distinct points.

 \begin{TEO}\label{green2}
 Let $g_1,  g_2, m$ be integers such that $ 4 \leq m \leq   \frac{g_1 +1}{2}  $ and $g_2\geq 1$.
 
 Let $C=C_1\cup C_2$ be a stable curve given by the union of  an irreducible  smooth  general
  curve  $C_1$  of genus $g_1$  and an irreducible smooth curve $C_2$   of genus $g_2 $,
 meeting transversally in $m$ distinct points
 $\{x_1, \cdots, x_m\}$.  
Then %$Cliff(C)=m-2$ and Green's conjecture holds for $C$.%, i.e.,  $\sK_{p,1}(C,\omega_C)=0$ iff $p\geq p_a(C)-\Cliff(C)-1$.
 $$\Cliff(C) = m-2 \ \  \mbox{ and }  \ \   \sK_{p,1}(C, \omega_C)= 0   \Longleftrightarrow  p \geq  p_a(C)  - \Cliff(C)-1.$$
 \end{TEO}

\begin{proof} Since $ p_a(C) = g_1+g_2 +m-1$, 
 %We have $\Cliff(C) \leq m-2$ by Prop. \ref{prop:Clifford} and furthermore  $ \sK_{p_a(C)-m,1}(C,\omega_C)\neq 0$ by Corollary \ref{K_p-decomposition}.
the theorem follows if we prove that %$\sK_{p,2}(\omega_C)=0$ for $p= m-3$, or equivalently,
$\sK_{p,1}(\omega_C)=0$ if and only if   $p \geq  p_a(C) - m+1 =g_1+g_2$.

\hfill\break First of all  notice that by Thm.  \ref{thm:curve}  the linear  system $|\omega_C|$  yields an embedding $\varphi: C \into \proj^{p_a(C)-1}$  such that  $\varphi (C)$ is the union of two curves 
of genus $g_1$ (resp. $g_2$) and degree $2g_1-2 +m$ (resp. $2g_2-2 +m$) 
 intersecting in $m$  points $\{  \varphi(x_1), \cdots , \varphi(x_m) \} $.
 % Moreover such points    span a $(m-2)$-dimensional subspace
%since for $i=1,2$
%$${\omega_C}_{| C_i} \iso {\omega}_{C_i} \otimes \Oh_{C_i}(x_1 + \cdots + x_m)$$
%and the map $H^0(C, \omega_C) \to \Oh_{(x_1 \cup  \cdots \cup  x_m)}$ has rank $m-1$ because  its kernel is isomorphic to
%$H^0(C_1, {\omega}_{C_1}) \oplus H^0(C_2, {\omega}_{C_2}) $.

Now consider 
the standard exact sequence
$$ 0 \to \Oh_{C_2} (-C_1 )\to\Oh_C \to
  \Oh_{C_1}\to 0 .$$
Twisting with $\omega_C^{\otimes q}$ and   taking cohomology we get  the following exact sequence of $S(H^0(C, \omega_C))$-modules 
\begin{equation}\label{long-koszul}
 0 \to \bigoplus_{q\geq 0} H^0(C_2, {\omega_C}_{| C_2}^{\otimes q}  (-C_1) )\to\bigoplus_{q\geq 0} H^0(C,\omega_C^{\otimes q} ) \to
  \bigoplus_{q\geq 0 } H^0(C_1, {\omega_C}_{| C_1}^{\otimes q} )) \to 0
  \end{equation}
  where the  maps preserve the grading. 
  
Let $W := H^0(C,\omega_C)$. 
  We emphasize that $\varphi(C_1)$ and    $\varphi(C_2)$ are embedded as   degenerate curves in  $ \proj (W^\vee)$,
   but we can still consider  every terms   above  as $S(W)$-modules. 
    Therefore we will use the notation   $ {\sK}_{p,q}\big(  - , - , W \big)$ 
   to point out  that we are  finding  the resolution of the ideal of such degenerate curves % , and hence we will consider  Koszul groups    of  $S(W)$-modules
    (see  \cite[Proof of  Thm. (3.b.7)]{Gr} for a similar argument). 
     
By  \cite[Corollary 1.4.d, Thm. 3.b.1 ]{Gr}  we have the following exact sequence of Koszul groups :

\begin{eqnarray}\nonumber
\cdots \to { \sK}_{p+1,0}\big(C_1,{\omega_C}_{| C_1}, W\big) \to
 { \sK}_{p,1}\big(C_2, \Oh_{C_2}(-C_1) , {\omega_C}_{| C_2}, W \big)  \to \\ \label{long-koszul2}  \to 
 {\sK}_{p,1}\big(C, \omega_C\big)  \to  
 {\sK}_{p,1}\big(C_1,{\omega_C}_{| C_1}, W\big)  \to \cdots  \phantom{aaaaaaaaaaaa}
 %\to  \\ \nonumber \to  { \sK}_{p-1,2}(C_2, \Oh_{C_2}(-C_1) , {\omega_C}_{| C_2}, W ) 
\end{eqnarray}
To deal with the above groups 
  we consider
 the splittings

$W = H^0(C, \omega_C)=H^0(C_2, {\omega_C}_{| C_2}) \bigoplus   U $  with $U \iso H^0(C_1, {\omega}_{C_1})$,

$W = H^0(C, \omega_C)=H^0(C_1,{\omega_C}_{| C_1}) \bigoplus   Z $  with $Z \iso H^0(C_2, {\omega}_{C_2})$.

\hfill\break
Setting $s = \max\{0, p-g_1 \}$,   $t = \max\{0, p-g_2 \}$ we have the following decompositions of the Koszul groups appearing in  the above exact sequence:

 $
\begin{array}{rl}
\hspace{ -6 mm}  {\sK}_{p,1}(C_2, \Oh_{C_2}(-C_1), {\omega_C}_{| C_2}, W \big)   & =  \  \  {\displaystyle \bigoplus_{s\leq p' \leq p}  \big[ \sK_{p',1}(C_2,\Oh_{C_2}(-C_1) , {\omega_C}_{| C_2} ) \   \otimes
      \stackrel{p-p'}{\bigwedge} U \big]} \\
        {\sK}_{p,1}(C_1,{\omega_C}_{| C_1}, W )    &  =  \  \  {\displaystyle  \bigoplus_{t\leq p'' \leq p}  \big[ \sK_{p'',1}(C_1, {\omega_C}_{| C_1})  \   \otimes
      \stackrel{p-p''}{\bigwedge} Z \big] }
      \end{array}$
Let us study at first 
${ \sK}_{p,1}(C_2, \Oh_{C_2}(-C_1) , {\omega_C}_{| C_2}, W ) $. 

%  $ {\widetilde{ \sK}}_{p,1}(C_2, \Oh_{C_2}(-C_1) , \omega_C ). $
Fix  $p' \leq p$. By duality  (cf. \cite[Prop. 1.4]{Fr}) and  the shift properties of $\sK_{p,q}$  \cite[(2.a.17)]{Gr}    we have the following isomorphisms
\begin{eqnarray*}\sK_{p',1}(C_2,\Oh_{C_2}(-C_1) , {\omega_C}_{| C_2} ) & \iso & \sK_{g_2+m-3-p',1}(C_2, {\omega}_{C_2} \otimes \Oh_{C_2}(C_1) , {\omega_C}_{| C_2} ) \\
& \iso & \sK_{g_2+m-3-p',2}(C_2 , {\omega_C}_{| C_2} ).
\end{eqnarray*}
 But    $\deg({\omega_C}_{|C_2}) = 2g_2-2 + m$. Whence by \cite[Theorem (4.a.1)]{Gr}
 %{\color{red}satisfies properties $N_k$} as long as $m-3 \geq  k$,  that is,
 $$  \sK_{g_2+m-3-p',2}(C_2 , {\omega_C}_{| C_2} )=0 \ \mbox{ if }  \ \ g_2+m-3-p' \leq m-3.   $$
 %Moreover by  a result of Green and Lazarsfeld (see \cite[Thm. 3.3]{GL2}) we have 
% $$\sK_{g_2+m-3-p',2}(C_2 , {\omega_C}_{| C_2} )\neq 0 \mbox{ if }  \ \ g_2+m-3-p' = m-2   $$
% since $C_1\cap C_2$ consists of  $m$ points spanning a $(m-2)$-dimensional subspace.
%
Therefore we get 
%since  $p-g_1 \leq p' \leq p$, we obtain
\begin{equation}\label{K_pC_2}
 { \sK}_{p,1}\big(C_2, \Oh_{C_2}(-C_1), {\omega_C}_{|C_2}, W \big)=0  \mbox{ if }   p \geq g_1 + g_2.
  \end{equation}

\hfill\break
  Now let us study $ {\sK}_{p,1}\big(C_1,{\omega_C}_{|C_1}, W \big). $

  By our assumption $C_1$ is a general curve  of  genus $g_1 \geq 2m -1$.  For a general curve  of genus $g_1$ we have $\Cliff(C_1) = [  \frac{g_1 - 1}{2}]$ 
  and 
by the results of Voisin on Green's conjecture for smooth curves with maximal Clifford index (\cite{V1}, \cite{V2}) we have $ \sK_{p,1}(C_1, {\omega}_{C_1})=0  \mbox{ if }  \ \ p \geq g_1 -  [  \frac{g_1 +1}{2}]$ .

 By our construction
  ${\omega_C}_{| C_1} \iso \omega_{C_1}(C_2) \iso \omega_{C_1}\otimes \Oh_{C_1}(x_1+\ldots +x_m)$, hence by the result of Aprodu on adjoint bundles 
  \cite[Thm. 3]{ap02}  we  get
$$ \sK_{p'',1}(C_1,\omega_C)=0  \mbox{ if }  \ \ p'' \geq   g_1 +  m -  [  \frac{g_1 + 1}{2}]$$
   and in particular 
\begin{equation}\label{K_pC_1}   {\sK}_{p,1}\big(C_1,{\omega_C}_{|C_1},  W \big) =0 \mbox{ if }  \ \ p \geq  g_2+g_1 + m - [ \frac{g_1 +1}{2}].
 \end{equation}
 %=  g_2+ g_1 +m -1 - [\frac{g_1+3}{2}]  $$
Therefore,   since $  m \leq  \frac{g_1 +1}{2} $ by our assumptions, we obtain %s (\ref{K_pC_2}) and  (\ref{K_pC_1}) yields 
$$ {\sK}_{p,1}\big(C_1,{\omega_C}_{|C_1},  W \big) =0 \mbox{ if }  \ \ p \geq  g_1+g_2. $$ 
Putting our vanishing results (\ref{K_pC_2}) and (\ref{K_pC_1}) into  the exact sequence  (\ref{long-koszul2})  we deduce that 
 $$ \sK_{p,1}(C, \omega_C)=0  \mbox{ if }  \ \  p \geq g_1 + g_2  = p_a(C)  - (m-2) -1 . $$
To conclude the proof notice  that the above vanishing result implies  $\Cliff(C) \geq m-2$ by  Proposition \ref{nonvanishing},   %. The same proposition yields 
%$\sK_{p,1}(\omega_C)\neq 0$ for $p \leq  p_a(C) - m =g_1+g_2 -1$ while 
     whereas we have $\Cliff(C) \leq m-2$ 
 by Thm. \ref{prop:Clifford}  because  $m-2 \leq   [\frac{p_a(C)-1}{2}]$ by our numerical assumptions.

 Therefore  it is 
$\Cliff(C) = m-2$ and  $\sK_{p,1}(C, \omega_C)= 0 $ if and only if  $ p \geq  p_a(C)  - \Cliff(C)-1.$

 \end{proof}

%\bibliographystyle{amsplain}   % this means that the order of references
			    % is dtermined by the order in which the
			    % \cite and \nocite commands appear
%\bibliography{biblio}

  \vspace{.5cm}
 \hfill\break 
Marco Franciosi\\
Dipartimento di Matematica, Universit\`a di Pisa\\
Largo B. Pontecorvo 5, I-56127 Pisa (Italy)\\
\tt{marco.franciosi@unipi.it}

\end{document}